\documentclass[11pt, a4paper]{article}

\usepackage{amsthm}
\usepackage{amsmath}
\usepackage{amssymb}
\usepackage{amscd}                 
\usepackage[all]{xy}               
\usepackage{stmaryrd}              
\usepackage{a4wide}
\usepackage[toc,page]{appendix}

\theoremstyle{plain}

\newtheorem{thm}{Theorem}[section]
\newtheorem{cor}[thm]{Corollary}
\newtheorem{lem}[thm]{Lemma}

\newtheorem{prop}[thm]{Proposition}

\theoremstyle{definition}
\newtheorem{defn}[thm]{Definition}

\newtheorem{rem}[thm]{Remark}
\newtheorem{remarks}[thm]{Remarks}

\newcommand\Q{{\mathbb Q}}

\newcommand\N{{\mathbb N}}

\newcommand{\Gal}{\mathop{\mathrm{Gal} }\nolimits}

\newcommand{\Sp}{\mathop{\mathrm{Sp} }\nolimits}
\newcommand{\GSp}{\mathop{\mathrm{GSp} }\nolimits}

\newcommand\F{\mathbb{F}}

\newcommand\Frob{\mathrm{Frob}}

\newcommand{\Spec}{\mathop{\mathrm{Spec} }\nolimits}

\newcommand{\Aut}{\mathop{\mathrm{Aut} }\nolimits}

\newcommand\fs{\mathfrak{s}}
\newcommand\ft{\mathfrak{t}}
\newcommand\sA{\mathcal{A}}

\newcommand\an{\textup{an}}
\renewcommand\O{\mathcal{O}}

\newcommand\romanenum{\renewcommand{\labelenumi}{\textup{(}\roman{enumi}\textup{)}}}

\title{Abelian varieties over number fields, tame ramification and big Galois image}

\date{}

\begin{document}

\romanenum
\author{Sara Arias-de-Reyna \and Christian Kappen}
\maketitle

\begin{abstract}
Given a natural number $n\geq 1$ and a number field $K$, we show the existence of an integer $\ell_0$ such that for any prime number $\ell\geq \ell_0$, there exists a finite extension $F/K$, unramified in all places above $\ell$, together with a principally polarized abelian variety $A$ of dimension $n$ over $F$
such that the resulting $\ell$-torsion representation 
\[
 \rho_{A,\ell}\,:\,G_F\rightarrow\GSp(A[\ell](\overline{F}))
\]
is surjective and everywhere tamely ramified. In particular, we realize $\GSp_{2n}(\F_\ell)$ as the Galois group of a finite tame extension of number fields $F'/F$ such that $F$ is unramified above $\ell$.
\end{abstract}

\section{Introduction}

Let $\ell$ be a prime number. In this paper, we establish results regarding the existence of abelian varieties $A$ over number fields $F$ such that the representation $\rho_{A,\ell}$ of the absolute Galois group $G_F$ of $F$ on the symplectic $\F_\ell$-vector space of geometric $\ell$-torsion points of $A$ satisfies certain local and global conditions. More specifically, we are interested in finding $A/F$ such that $\rho_{A,\ell}$ is everywhere tamely ramified and has large image.

Our interest in such objects $A/F$ was inspired by the tame inverse Galois problem over the rational field $\Q$: if one could construct $A/\Q$ of dimension $n\geq 1$ such that
\begin{enumerate}
\item $\rho_{A, \ell}$ is everywhere tamely ramified and
\item $\rho_{A, \ell}:G_\Q\rightarrow\GSp_{2n}(\F_\ell)$ is surjective,
\end{enumerate}
then the finite Galois extension $(\overline{\Q})^{\ker \rho_{A, \ell}}/\Q$ would be tamely ramified with Galois group $\GSp_{2n}(\mathbb{F}_{\ell})$, thus providing a solution to the tame inverse Galois problem over $\Q$ for the group $\GSp_{2n}(\mathbb{F}_{\ell})$. Let us note that tameness is only a condition at the primes $p$ dividing the order of $\GSp_{2n}(\F_{\ell})$. Let us further note that if $p$ is a prime different from $\ell$ such that $A$ has good reduction at $p$, then $\rho_{A,\ell}$ is unramified (and hence tame) at $p$, by the criterion of N\'eron-Ogg-Shafarevich.

In previous work, the first author has constructed, for each prime number
$\ell$, infinitely many nonisomorphic non-CM elliptic curves over
$\Q$ having good supersingular reduction at $\ell$ (see
\cite{AriasVila1}); these elliptic curves satisfy the two conditions listed above. For abelian varieties of arbitrary dimension with good supersingular reduction at $\ell$, she has isolated a condition, called Hypothesis (H), which is sufficient to imply tame ramification at $\ell$ (see \cite{Arias}). Hypothesis (H) is a condition on the valuations of the coordinates of the $\ell$-torsion points of the group attached to the formal group law of the abelian variety. Using very explicit equations she manages to obtain, for each $\ell$, infinitely many abelian surfaces over $\Q$ arising as Jacobians of suitable curves, satisfying Hypothesis (H) and having trivial absolute endomorphism ring (see \cite{AriasVila2}). Her construction proceeds roughly as follows: she first finds a suitable genus two curve locally at $\ell$ which, in a second step, is globalised; the resulting curve over $\Q$ is then deformed in order to also satisfy requirements at primes different from $\ell$, requirements which ensure certain global properties (e.g.\ large image of the resulting Galois representation).

This approach to the problem relies on very explicit computations on the Jacobians of genus two curves, methods which are not available in higher dimensions. To address the question whether there exist higher-dimensional abelian varieties with the property that the Galois representation on the $\ell$-torsion points is tame and has large image, we follow a more conceptual approach: we consider a suitable moduli space of abelian varieties, and we prove the existence of a point $P$ on this space defined over some number field unramified above $\ell$ such that $P$ satisfies certain local conditions at a finite number of finite places, properties that will ensure tame ramification and large image. Our main result is the following:

\begin{thm}\label{mainthm}
Given a number field $K$ and an integer $n \geq 1$, there exists an integer $\ell_0$ such
that, for all prime numbers $\ell\geq \ell_0$, there exist a finite extension $F$ of $K$, unramified in all places above $\ell$, and an $n$-dimensional abelian variety
$A$ defined over $F$ such that the associated $\ell$-torsion representation
$\rho_{A, \ell}: G_F \rightarrow \GSp(A[\ell](\overline{F}))$
is surjective and everywhere tamely ramified.
\end{thm}

As a corollary, we obtain, for almost all prime numbers $\ell$, the group $\GSp_{2n}(\F_\ell)$ as the Galois group of a finite extension $F'/F$ of number fields, where $F$ is unramified above $\ell$. Let us note, however, that this result is a special case of \cite{Cal2011} Prop.\ 3.2, which solves the potential tame inverse Galois problem in full generality, i.e.\ for arbitrary finite groups, even with the imposition of local conditions at a finite number of primes. Our construction nevertheless provides a tamely ramified surjective Galois representation 
which comes from the $\ell$-torsion of a suitable abelian variety defined over a number field and, hence, is a member of a compatible system of Galois representations.

Let us outline the structure of the present paper. In Section \ref{surjsec} we give, for $\ell\geq 3$, a criterion for the $\ell$-torsion representation $\rho_{A,\ell}$ of an $n$-dimensional abelian variety defined over a number field $K$ to be surjective onto $\GSp_{2n}(\F_\ell)$:
we find that $\rho_{A,\ell}$  is surjective if the degree of the field extension $K(\mu_{\ell})/K$ obtained by adjoining to $K$ the $\ell$-th roots of unity equals $\ell-1$ and if furthermore the image of $\rho_{A, \ell}$ contains a transvection as well as an element whose characteristic polynomial is irreducible and has nonzero trace. (In the appendix, we show that such elements always exist in $\GSp_{2n}(\F_{\ell})$ if $\ell\geq 5$ and $\ell \nmid n$). In Section \ref{toolsec}, we revisit some tools from arithmetic geometry which we will need in the proof of Theorem \ref{mainthm}, namely Moret-Bailly's theorem on the existence of global points and Kisin's results on local constancy in $p$-adic families of Galois representations. In Section \ref{sec:new}, we set up some notation, and we apply the tools from the previous section to develop a series of useful results about existence of abelian varieties with prescribed properties that will be combined in the proof of the main result. In Section \ref{proofsec}, we give the proof of Theorem \ref{mainthm}, and we mention some open questions. Let us briefly outline the strategy of the proof of Theorem \ref{mainthm}:

Using the appendix of this paper and results of Hall and Kowalski (\cite{Hall}), we find, after possibly extending $K$, an integer $\ell_0>j$ and a Jacobian $J$ over $K$ admitting a full symplectic level $j$ structure (for some prefixed integer $j\geq 3$) such that $\mu_j\subseteq K$ and such that for all prime numbers $\ell\geq\ell_0$, $K/\Q$ is unramified at $\ell$, and the image of the $\ell$-torsion representation $\rho_{J,\ell}$ of $J$ contains an irreducible element of nonzero trace $\fs$ and a transvection $\ft$. Let $K''=K(J[\ell], J[j])$, and let $S$ denote the set of finite places of $K$ dividing the order of $\GSp_{2n}(\F_\ell)$. We find finite extensions $(K'_\nu/K_\nu)_{\nu\in S}$ of common degree $r$ together with elliptic curves of good reduction $E_\nu/K'_\nu$, defined over $\Q$, admitting full symplectic level $j$ structures, where for $\nu|\ell$ the reduction of $E_\nu$ is in addition supersingular. We then find a finite extension $K'/K$ of degree $r$ that is linearly disjoint from $K''/K$ and that induces the local extensions $K'_\nu/K_\nu$. After replacing $K$ by $K'$, the elliptic curves $E_\nu$ with their level structures are defined over the $K_\nu$, and the elements $\fs$ and $\ft$ still lie in the image of $\rho_{J,\ell}$. By \v{C}ebotarev's density theorem, there exist finite places $\mu_\fs$ and $\mu_\ft$ of $K$ away from $S$ such that $\mathfrak{r}\in\{\fs,\ft\}$ is the $\rho_{J,\ell}$-image of a Frobenius element at $\mu_\mathfrak{r}$. We set $S'=S\cup\{\mu_\fs,\mu_\ft\}$. Let $\sA=\sA_{n,1,j/K}$ denote the $K$-variety parametrizing principally polarized $n$-dimensional abelian varieties with full level $j$ structure. For each $\nu\in S$, let $x_\nu\in \sA(K_\nu)$ denote the point defined by $E_\nu^n$, and for $\mathfrak{r}\in\{\fs,\ft\}$, let $x_{\mu_\mathfrak{r}}\in\sA(K_\nu)$ denote the point defined by $J\otimes_KK_{\mu_\mathfrak{r}}$. By Kisin's results on local constancy of families of Galois representations (cf.\ \cite{Kisin1999}), there exists, for each $\nu\in S'$, a $\nu$-adically open neighborhood $\Omega_\nu$ of $x_\nu$ such that for each $y_\nu\in\Omega_\nu$, the resulting representations $\rho_{x_\nu,\ell}$ and $\rho_{y_\nu,\ell}$ coincide up to conjugation by elements in $\GSp_{2n}(\F_\ell)$. By a result of Moret-Bailly (cf.\ \cite{moret-bailly}), there exists a finite field extension $F/K$ linearly disjoint to $K(\mu_\ell)/K$ together with a point $y\in\sA(F)$ such that the places $\nu\in S'$ split completely in $F$ and such that $y_\nu\in\Omega_\nu$ for all $\nu\in S'$. The abelian variety $A/F$ defined by $y$ has then the property that $\rho_{A,\ell}$ is tamely ramified in all places of $F$ dividing the order of $\GSp_{2n}(\F_\ell)$ and that the image of $\rho_{A,\ell}$ contains conjugates of $\fs$ and $\ft$. Now $\rho_{A,\ell}$ is everywhere tamely ramified, and $\rho_{A,\ell}$ maps onto $\GSp_{2n}(\F_\ell)$. All the extensions of $K$ that we made in the above proof can be chosen to be unramified in the places above $\ell$, and $F/K$ is totally split in the places above $\ell$; hence $F$ is unramified above $\ell$, as desired.

Let us remark that if we dropped the requirement on $F$ of being unramified above $\ell$, there would be an easier way to find $A/F$: we could then simply consider the Jacobian $J/K$ above and eliminate ramification in primes dividing the order of $\GSp_{2n}(\F_\ell)$ by means of a finite extension $F/K$ that is orthogonal to $K(J[\ell])$.

The first author worked on this project as a research fellow of the Alexander von Humboldt
Foundation; she was partially supported by the Ministerio de Educaci\'on y
Ciencia grant MTM2009-07024.
The second author was partially supported by the SFB 45 ``Periods, moduli spaces and arithmetic of algebraic varieties'', and he would like to thank the University of Luxembourg for its hospitality. Both authors would like to thank Brian Conrad, Ulrich G\"ortz and Gabor Wiese for helpful discussions. They would also like to thank the anonymous referee, for his useful remarks and his very helpful suggestions regarding the overall structure of the paper.

\section{Surjectivity of representations attached to abelian varieties}\label{surjsec}

Let $\ell\geq 3$ be a prime number, let $n$ be a positive integer, and let $V$ be a $2n$-dimensional vector space over $\F_{\ell}$ endowed with a symplectic form. Let us recall the following result (cf. Theorem 2 of \cite{LiZha}, Theorem 1.1 of \cite{ArDiWiII}):

\begin{thm}\label{Principal} Let $G\subset \GSp(V)$ be a subgroup containing a transvection. Then exactly one of the following statements holds:
\begin{enumerate}
 \item The action of $G$ on $V$ is reducible in the sense that it stabilizes a nontrivial nonsingular symplectic subspace.
\item There exists a proper decomposition $V=\bigoplus_{i\in I} S_i$ of $V$ into equidimensional nonsingular symplectic subspaces $S_i$ such that for each $g\in G$ and each $i\in I$, there exists some $j\in I$ with $g(S_i)\subseteq S_j$ and such that the resulting action of $G$ on $I$ is transitive.
 \item $G\supset \Sp(V)$.
\end{enumerate}
\end{thm}

From this statement we can derive a useful criterion ensuring that a group $G\subset \GSp(V)$ contains $\Sp(V)$:

\begin{cor}\label{cor:hugeimage}
Let $G\subset \GSp(V)$ be a subgroup containing a transvection and an element of nonzero trace whose characteristic polynomial is irreducible. Then $G\supset \Sp(V)$.
\end{cor}

\begin{proof}
Since $G$ contains a transvection, we are in one of the three cases of Theorem \ref{Principal}. Let $g\in G$ be an element of nonzero trace whose characteristic polynomial is irreducible. First of all, $g$ acts irreducibly on $V$, and hence the action of $G$ on $V$ cannot be reducible, which excludes case ($i$). Assume now that we are in case ($ii$), i.e.\ that $G$ preserves a nontrivial decomposition $V = \bigoplus_{i=1}^h S_i$. Since $g$ acts irreducibly, it cannot fix any of the $S_i$; hence the trace of $g$ must be zero, which leads to a contradiction. We conclude that only case ($iii$) of Theorem \ref{Principal} is compatible with the existence of $g$.
\end{proof}

\begin{rem}
In the Appendix, we prove that if $n$ is any positive integer and $\ell$ is an odd prime, then if $r\in \N$ is such that $\ell^r$ is sufficiently large, then the polynomial ring $\F_{\ell^r}[X]$ contains an irreducible symplectic polynomial of nonzero-trace and of degree $2n$.
\end{rem}

If $K$ is a number field, if $\overline{K}$ is an algebraic closure of $K$ and if $A$ is a principally polarized abelian $K$-variety of dimension $n$, we let $A[\ell](\overline{K})$ denote the $\F_{\ell}$-vector space of geometric $\ell$-torsion points of $A$; it is a $2n$-dimensional vector space over $\mathbb{F}_{\ell}$, and the Weil pairing gives rise to a non-degenerate symplectic form $\langle \cdot, \cdot \rangle:A[\ell]\times A[\ell]\rightarrow \mu_{\ell}$, where $\mu_{\ell}$ denotes the finite \'etale $K$-group scheme of $\ell$-th roots of unity. By functoriality, the action of $G_K=\mathrm{Gal}(\overline{K}/K)$ on $A[\ell](\overline{K})$ and on $\mu_\ell(\overline{K})$ commutes with this pairing. That is,
if we let $\chi_{\ell}$ denote the mod $\ell$ cyclotomic character, then for all $\sigma\in G_K$ and for all $P_1, P_2 \in  A[\ell](\overline{K})$,
\begin{equation*} \langle P_1^{\sigma}, P_2^{\sigma}\rangle=\chi_{\ell}(\sigma) \langle P_1, P_2\rangle\;,\end{equation*}
which compels the image of the representation $\rho_{A,\ell}:G_K\rightarrow\Aut(A[\ell](\overline{K}))$ to be contained in the general symplectic group $\GSp(A[\ell](\overline{K}))\simeq\mathrm{GSp}_{2n}(\mathbb{F}_{\ell})$. If  $[K(\mu_{\ell}):K]=\ell-1$, then $\chi_\ell$ is  surjective. In this case, if the image of $\rho_{A, \ell}$ contains $\Sp(A[\ell](\overline{K}))$, it must already coincide with $\GSp(A[\ell](\overline{K}))$.

Since a number field contains only a finite number of roots of unity, this observation
implies the following result:

\begin{lem}\label{mainsurjprop}
Let $K$ be a number field. Then there exists an integer $\ell_0$ such that for all primes $\ell\geq \ell_0$ and all abelian varieties $A/K$, the following holds: if the image of $\rho_{A, \ell}$ contains $\Sp_{2n}(A[\ell](\overline{K}))$, then the image of $\rho_{A, \ell}$ already coincides with $\GSp(A[\ell](\overline{K}))$.
\end{lem}

\section{Compilation of tools}\label{toolsec}

In this section, we discuss some tools from arithmetic geometry that we will use in the proof of our main theorem. We do so mainly for the convenience of the reader, but also to provide some complementary details.

\subsection{Local constancy of $p$-adic families of Galois representations}\label{toolseclocal}

In order to discuss Kisin's results on local constancy in $p$-adic families of Galois representations (cf.\ \cite{Kisin1999}), we need the language of \'etale fundamental groups. We refer to \cite{sga1} Exp.\ V \S 4-7 for generalities on fundamental groups and in particular on \'etale fundamental groups for schemes. We will also need to use \'etale fundamental groups of rigid-analytic spaces, by which, following Kisin \cite{Kisin1999}, we mean the \emph{algebraic} \'etale fundamental groups defined in $\cite{dejong_fundgrp}$. If $U$ is a connected scheme or a connected rigid space and if $\bar{x}$ is a geometric point of $U$, we let $F_{U,\bar{x}}$ denote the associated fiber functor from the category of finite \'etale $U$-objects to the category of finite sets. By \cite{sga1} Exp.\ V Cor.\ 5.7, any two such fiber functors are isomorphic. If $\bar{y}$ is another geometric point of $U$, then an isomorphism of fiber functors $F_{U,\bar{y}}\cong F_{U,\bar{x}}$ will also be called an \emph{\'etale path} from $\bar{y}$ to $\bar{x}$ inside $U$. Since the \'etale fundamental group $\pi_1(U,\bar{x})$ of $U$ at $\bar{x}$ is, by definition, the group of automorphisms of the functor $F_{U,\bar{x}}$, any \'etale path $F_{U,\bar{x}}\cong F_{U,\bar{y}}$ gives rise to an isomorphism $\pi_1(U,\bar{x})\cong\pi_1(U,\bar{y})$, and the \'etale paths $\bar{x}\sim\bar{x}$ correspond to the inner automorphisms of $\pi_1(U,\bar{x})$.

Let $K$ be a field that is complete with respect to a nontrivial nonarchimedean valuation; in the following, we will also write $\Spec K$ instead of $\Sp K$, by abuse of notation. Let $U$ be a connected $K$-scheme or a connected rigid $K$-variety, let $p$ denote the structural morphism from $U$ to $\Spec K$, and let us assume that $U$ admits $K$-rational points
\[
x,y\,:\,\Spec K\rightarrow U\;.
\]
We fix an embedding of $K$ into an algebraic closure $\overline{K}$ of $K$, we let $\overline{z}$ denote the corresponding geometric point of $\Spec K$, and we let $\bar{x},\bar{y}$ denote the resulting geometric points of $U$ above $x$ and $y$ respectively; then $\bar{x}$ and $\bar{y}$ map to $\overline{z}$ via the structural morphism $U\rightarrow\Spec K$. The morphism $p$ induces identifications of fiber functors
\[
F_{\bar{z}}\,=\,F_{U,\bar{x}}\circ p^*\,=\,F_{U,\bar{y}}\circ\,p^*\;,\quad(1)
\]
while $x$ and $y$ induce identifications
\begin{eqnarray*}
F_{U,\bar{x}}&=&F_{\bar{z}}\circ x^*\;\textup{and}\\
F_{U,\bar{y}}&=&F_{\bar{z}}\circ y^*\;;\quad\quad\quad\quad\quad(2)
\end{eqnarray*}
here $p^*$, $x^*$ are $y^*$ denote the respective pullback functors on categories of finite \'etale schemes. By ($1$), an \'etale path $\alpha:F_{U,\bar{y}}\cong F_{U,\bar{x}}$ induces an automorphism $\beta$ of $F_{\bar{z}}$; by ($2$), $\beta$ induces an automorphism $\gamma$ of $F_{U,\bar{x}}$. After modifying $\alpha$ by $\gamma$, we see:
\begin{lem}\label{etpathlem}
There exists an \'etale path $F_{U,\bar{y}}\cong F_{U,\bar{x}}$ which induces the identity on $F_{\bar{z}}$.
\end{lem}
An \'etale path $\phi:F_{U,\bar{y}}\rightarrow F_{U,\bar{x}}$ as in Lemma \ref{etpathlem} above induces a commutative diagram
\[
\xymatrix{
1\ar[r]&\pi_{1,\textup{geom}}(U,\bar{y})\ar[r]\ar[d]^{\varphi'}&\pi_1(U,\bar{y})\ar[r]\ar[d]^\varphi&\Gal(\overline{K}/K)\ar[r]\ar[d]^{\overline{\varphi}}\ar@/_1pc/[l]_{y}&1\\
1\ar[r]&\pi_{1,\textup{geom}}(U,\bar{x})\ar[r]&\pi_1(U,\bar{x})\ar[r]&\Gal(\overline{K}/K)\ar[r]\ar@/_1pc/[l]_{x}&1\;,\\
}
\]
where the vertical maps are isomorphisms, where the geometric \'etale fundamental groups on the left are defined so that the rows in the diagram are exact, where we have identified $\pi_1(\Spec K,\bar{z})$ with $\Gal(\overline{K}/K)$ and where $\bar{\varphi}$ is the identity on $\Gal(\overline{K}/K)$.
We then have a commutative diagram
\[
\xymatrix{
&F_{U,\bar{y}}(\cdot)\ar[r]^\phi_\sim\ar@(dl,dr)[]&F_{U,\bar{x}}(\cdot)\ar@(dl,dr)[]&\\
\quad\quad\quad\quad&\pi_1(U,\bar{y})\ar[r]^{\varphi}_\sim\ar[d]&\pi_1(U,\bar{x})\ar[d]&\quad\quad\quad\quad(3)\\
&\Gal(\overline{K}/K)\ar@{=}[r]\ar@/^/[u]^{y}&\Gal(\overline{K}/K)\ar@/^/[u]^{x}\;.&\\
}
\]
Let now $A$ be a principally polarized abelian $U$-scheme, and let $\ell$ be a prime number; then $A[\ell]$ and $\mu_\ell$ are finite \'etale $U$-schemes, and the symplectic structures on the geometric fibers of $A[\ell]$ are induced from morphisms of finite \'etale $U$-schemes
\begin{eqnarray*}
+&:&A[\ell]\times A[\ell]\rightarrow A[\ell]\quad\textup{and}\\
\langle\cdot,\cdot\rangle&:&A[\ell]\times A[\ell]\rightarrow \mu_\ell\;.
\end{eqnarray*}
Since $\pi_1(U,\bar{x})$ is the automorphism group of the \emph{functor} $F_{U,\bar{x}}$, it thus acts on $F_{U,\bar{x}}(A[\ell])$ via symplectic automorphisms, that is, the image of the natural map
\[
\rho_{\bar{x}}\,:\,\pi_1(U,\bar{x})\rightarrow\Aut(F_{U,\bar{x}}(A[\ell]))
\]
lies in $\GSp(F_{U,\bar{x}}(A[\ell]))$, and the analogous statement holds for $\bar{y}$. Similarly, since $\phi$ is an isomorphism of fiber \emph{functors}, the isomorphism $\phi(A[\ell]):F_{U,\bar{y}}(A[\ell])\overset{\sim}{\rightarrow}F_{U,\bar{x}}(A[\ell])$ respects symplectic structures, so we obtain a commutative diagram
\[
\xymatrix{
&\GSp(F_{U,\bar{y}}(A[\ell]))\ar[r]^\phi_\sim&\GSp(F_{U,\bar{x}}(A[\ell]))&\\
\quad\quad\quad\quad&\pi_1(U,\bar{y})\ar[r]^{\varphi}_\sim\ar[d]\ar[u]&\pi_1(U,\bar{x})\ar[d]\ar[u]&\quad\quad\quad\quad(4)\\
&\Gal(\overline{K}/K)\ar@{=}[r]\ar@/^/[u]^{y}&\Gal(\overline{K}/K)\ar@/^/[u]^{x}\;.&\\
}
\]

Let now $S$ be a connected $K$-scheme of finite type, let us assume that $K$ has positive residue characteristic, and let $x$ be a $K$-rational point of $S$. If $X$ is any finite \'etale $S$-scheme, then by a main result of Kisin's article \cite{Kisin1999}, there exists an admissible open neighborhood $U$ of $x$ in the rigid analytification $S^\an$ of $S$ such that the action of $\pi_1(U,\bar{x})$ on the fiber $F_{U,\bar{x}}(X)$ factors through the section $x$. In other words, if
\[
\rho_{\bar{x}}\,:\,\pi_1(U,\bar{x})\rightarrow\Aut(F_{U,\bar{x}}(X))
\]
is the natural homomorphism, if $\pi_{\bar{x}}$ is the natural projection from $\pi_1(U,\bar{x})$ onto $\Gal(\overline{K}/K)$ and if $x$ is its given section, then $\rho_{\bar{x}}=\rho_{\bar{x}}\circ x\circ \pi_{\bar{x}}$. Let now $A$ be an abelian scheme over $S$;
we choose an admissible open neighborhood $U$ of $x$ in $S^\an$ such that $U$ satisfies the conclusion of Kisin's theorem for both $A[\ell]$ and $\mu_\ell$; this is possible since intersections of admissible open subspaces of $S^\an$ are again admissible open in $S^\an$. Then for any $K$-rational point $y$ of $U$ and any \'etale path $\phi:F_{\bar{y}}\overset{\sim}{\rightarrow} F_{\bar{x}}$ inside $U$ satisfying the condition of Lemma \ref{etpathlem}, we obtain a commutative diagram
\[
\xymatrix{
\GSp(F_{\bar{y}}(A[\ell]))\ar[r]^\sim&\GSp(F_{\bar{x}}(A[\ell]))\\
\Gal(\overline{K}/K)\ar@{=}[r]\ar[u]^y&\Gal(\overline{K}/K)\ar[u]^x\;.
}
\]
We conclude:
\begin{thm}\label{kisinthm}
If $K$ is a field that is complete with respect to a nontrivial nonarchimedean valuation of positive residue characteristic, if $S$ is a connected $K$-scheme of finite type, if $x$ is a $K$-rational point of $S$, if $A$ is a principally polarized abelian $S$-scheme and if $\ell$ is any prime number, then there exists an admissible open neighborhood $U$ of $x$ in the rigid analytification $S^\an$ of $S$ such that for all $K$-rational points $y$ of $U$ and for any choice of (symplectic) $\F_{\ell}$-bases of $A[\ell]_x(\overline{K})$, $A[\ell]_y(\overline{K})$ and $\mu_\ell(\overline{K})$, the natural representations
\[
\rho_x,\rho_y\,:\,G_K\rightarrow\GSp_{2n}(\F_\ell)
\]
coincide up to conjugation by an element of $\GSp_{2n}(\F_\ell)$, where $n$ denotes the dimension of $A$.
\end{thm}

\subsection{Existence of global points}\label{toolsecglobal}

We quote a special case of a theorem of Moret-Bailly, cf.\ \cite{moret-bailly} and \cite{Cal2011} Theorem 3.1:

\begin{thm}\label{mbthm}
Let $K$ be a number field, let $X$ be a smooth geometrically connected $K$-variety, let $K'/K$ be a finite field extension, let $S$ be a finite set of (possibly infinite) places of $K$, and for each $\nu\in S$, let $\Omega_\nu$ be a nonempty $\nu$-adically open subset of $X(K_\nu)$. Then there exists a finite extension $F/K$ that is linearly disjoint to $K'/K$, together with an $F$-rational point $x\in X(F)$ such that for each $\nu\in S$
\begin{enumerate}
\item the place $\nu$ splits completely in $F$, and
\item for every place $\tilde{\nu}$ of $F$ above $\nu$, $x_{\tilde{\nu}}\in\Omega_\nu$ via the resulting natural identification
\[
X(F_{\tilde{\nu}})\cong X(K_\nu)\;.
\]
\end{enumerate}
\end{thm}

\begin{remarks}\label{someremarks} Let us note:
\begin{enumerate}
\item If $X$ is as in Theorem \ref{mbthm} above and if $U\subseteq X_{K_\nu}^\an$ is an admissible open subset of the analytification of $X_{K_\nu}$, then $U(K_\nu)\subseteq X(K_\nu)$ is $\nu$-adically open.
\item Let $j\geq 3$ be an integer, and let $K$ be a number field containing the $j$-th roots of unity; then the moduli space $\sA_{n,1,j/K}$ of principally polarized $n$-dimensional abelian schemes with full symplectic level $j$ structure above $K$ is a smooth and geometrically connected $K$-scheme, cf.\ \cite{faltings_chai} Chap.\ IV Def.\ 6.1, Rem.\ 6.2 (c) and Cor.\ 5.10.
\end{enumerate}
\end{remarks}

\section{Existence of abelian varieties}\label{sec:new}

\subsection{Existence of $\ell$-torsion approximations}

We first apply the techniques of Sections \ref{toolseclocal} and \ref{toolsecglobal} to the problem of finding abelian varieties $A$ defined over finite extensions $F$ of a given number field $K$ such that, locally at the places of $F$ lying above a finite number of finite places $\nu$ of $K$, the $\ell$-torsion of $A$ coincides with the $\ell$-torsion of given abelian varieties that are defined over the local fields $K_\nu$.

\begin{defn}
Let $K$ be a number field, and let $n\geq 1$, $j\geq 1$ be integers. 
A local AV-datum (with level $j$-structure) of dimension $n$ over $K$ is a finite set $S$ of finite places of $K$ together with a family $(A_\nu\,;\,\nu\in S)$, where $A_\nu$ is an $n$-dimensional principally polarized abelian variety over $K_\nu$ (with full symplectic level $j$ structure). By abuse of notation, such a local AV-datum (with level $j$-structure) will be simply denoted by $(A_\nu\,;\,\nu\in S)$.
\end{defn}

\begin{defn}
Let $K$ be a number field, and let $(A_\nu\,;\,\nu\in S)$ be a local AV-datum of dimension $n$ over $K$. An $\ell$-torsion approximation of $(A_\nu\,;\,\nu\in S)$ is a pair $(F,A)$, where $F/K$ is a finite field extension that is totally split in $S$ and where $A$ is a principally polarized abelian $F$-variety such that for each $\nu\in S$, for each place $\tilde{\nu}$ of $F$ above $\nu$ and for any choice of symplectic $\F_\ell$-bases, the representations
\[
\rho_{A_\nu,\ell}\;\;\textup{and}\;\;\rho_{A,\ell|_{D_{\tilde{\nu}}}}\,:\, G_{F_{\tilde{\nu}}}\rightarrow\GSp_{2n}(\F_\ell)
\]
are conjugate.
\end{defn}

As we shall now see, the results of Kisin and Moret-Bailly show that $\ell$-torsion approximations $(F,A)$ of local AV-data with level j structure exist whenever $j\geq 3$ and $K\supset\mu_j$, even if one requires $F/K$ to be linearly disjoint from a given finite extension field $K'/K$. Heuristically speaking, Kisin's results show that a local AV-datum gives rise, via its $\ell$-torsion representations, to a family of congruence conditions on the moduli space, while Moret-Bailly's theorem tells us that these congruence conditions admit a global solution.

\begin{thm}\label{thm:step1}
Let $K$ be a number field, let $j\geq 3$ be an integer such that $\mu_j(\overline{K})\subseteq K$, and let $(A_\nu\,;\,\nu\in S)$ be a local AV-datum with level $j$-structure over $K$. Then for each prime number $\ell$ and for each finite field extension $K'/K$, there exists an $\ell$-torsion approximation $(F,A)$ of $(A_\nu\,;\,\nu\in S)$ such that $F/K$ is linearly disjoint to $K'/K$.
\end{thm}
 
\begin{proof} Let $n$ denote the dimension of $(A_\nu\,;\,\nu\in S)$, and let us write $\sA$ to denote the moduli space $\sA_{n,1,j/K}$ of $n$-dimensional principally polarized abelian schemes with full symplectic level $j$ structure over $K$, cf.\ Remark \ref{someremarks} ($ii$) above. For each $\nu\in S$, let  $x_\nu$ denote the $K_\nu$-valued point of $\sA$ corresponding to $A_{\nu}$.
By Kisin's theory (cf.\ Theorem \ref{kisinthm} above), we may choose, for each $\nu\in S$, a $\nu$-adically open neighborhood $\Omega_\nu$ of $x_\nu$ in $\sA(K_\nu)$ such that for any $y_\nu\in\Omega_\nu$, the representations $\rho_{x_\nu}:G_{K_{\nu}}\rightarrow \GSp(A_{\nu}[\ell](\overline{K}_{\nu}))\simeq \GSp_{2n}(\mathbb{F}_{\ell})$ and $\rho_{y_\nu}:G_{K_{\nu}}\rightarrow \GSp(B_{\nu}[\ell](\overline{K}_{\nu}))\simeq \GSp_{2n}(\mathbb{F}_{\ell})$ coincide up to conjugation with an element of $\GSp_{2n}(\F_\ell)$, where $B_\nu$ denotes the abelian $K_\nu$-variety that is given by $y_\nu$. By Moret-Bailly's theorem (cf.\ Theorem \ref{mbthm} above), there exists a finite extension $F/K$ linearly disjoint to $K'/K$ with a point $y\in\sA(F)$ such that each $\nu\in S$ splits completely in $F$ and such that for each $\nu\in S$ and each place $\tilde{\nu}$ of $F$ lying above $\nu$, the localization $y_{\tilde{\nu}}\in\sA(F_{\tilde{\nu}})$ of $y\in\sA(F)$ lies in $\Omega_\nu$, where we use the natural isomorphism $\sA(F_{\tilde{\nu}})\cong\sA(K_\nu)$ to regard $\Omega_\nu$ as a subset of $\sA(F_{\tilde{\nu}})$. The point $y$ now defines a principally polarized abelian variety $A$ over $F$ such that for all $\nu\in S$ and all places $\tilde{\nu}$ of $F$ above $\nu$,  $\rho_{A,\ell}|_{D_{\tilde{\nu}}}$ and $\rho_{A_{\nu}, \ell}$ coincide up to conjugation by an element of $\mathrm{GSp}_{2n}(\mathbb{F}_{\ell})$. 
\end{proof}

\subsection{Local conditions enforcing big image}

\begin{defn}
Let $K$ be a number field, let $A$ be a principally polarized abelian $K$-variety, and let $\ell$ be a prime number. We say that $A$ has big image at $\ell$ if the natural homomorphism
\[
\rho_{A,\ell}\,:\,G_K\rightarrow\GSp(A[\ell](\overline{K}))
\]
is surjective.
\end{defn}

Given an abelian variety $A$ defined over a field $K$, we will denote by $K(A[\ell])$ the finite Galois extension of $K$ defined by the kernel of $\rho_{A,\ell}$.

\begin{lem}\label{cebotarevlem}
Let $K$ be a number field, let $A$ be a principally polarized abelian $K$-variety with big image at a prime number $\ell$, let $S$ be any finite set of finite places of $K$, and let $C\subseteq\GSp(A[\ell](\overline{K}))$ be a conjugacy class. Then there exists a finite place $\nu$ of $K$ away from $S$ such that $\rho_{A,\ell}$ is unramified at $\nu$ and such that $\rho_{A,\ell}(\Frob_\nu)\in C$. If moreover $j\geq 1$ is an integer such that $K(A[\ell])$ and $K(A[j])$ are linearly disjoint over $K$, there exists such a place $\nu$ with the additional property that $A\otimes_KK_\nu$ admits a full symplectic level $j$ structure.
\end{lem}
\begin{proof}
Let $L=K(A[\ell])$, let $K'$ denote the field $K(A[j])$ if we are given an integer $j\geq 1$ as in the statement, and let us set $K'=K$ otherwise. Then in both bases, $L$ and $K'$ are linearly disjoint over $K$. It suffices to find a positive density set of finite places $\nu$ of $K$ such that $L/K$ is unramified at $\nu$, such that $\rho_{A,\ell}(\Frob_\nu)\in C$ and such that $\nu$ splits completely in $K'$. By \cite{neukirch_ant} Lemma 13.5, it thus suffices to find a positive density set of finite places $\nu$ of $K$ such that $\nu$ is unramified in $LK'$ (hence in $L$ and $K'$), such that $\rho_{A,\ell}(\Frob_\nu)\in C$ and such that $\Frob_\nu$ restricts to the trivial automorphism of $K'/K$. By the \v{C}ebotarev density theorem (cf.\ \cite{neukirch_ant} Thm.\ 13.4), it thus suffices to show that inside $\Gal(LK'/K)$, the $\rho_{A,\ell}$-preimage of $C$ and $\Gal(LK'/K')$ have nonempty intersection. This, however, follows from the fact that the restriction of $\rho_{A,\ell}$ to $G_{K'}$ is surjective, which in turn follows from the fact that $\rho_{A,\ell}$ is surjective and from the fact that $L$ and $K'$ are $K$-linearly disjoint.
%
\end{proof}

\begin{prop}\label{localsurjprop}
Given a number field $K$ and integers $n,j\geq 1$, there exist an integer $\ell_0\geq 1$ such that for every prime $\ell\geq\ell_0$, there exists a finite field extension $K''/K$ such that for every finite field extension $K'/K$ that is linearly disjoint from $K''/K$, there exists a set $S$ of two finite places of $K'$ away from any given finite set of finite places of $K'$ and a local AV-datum with level $j$ structure $(J_\nu\,;\,\nu\in S)$ such that every $\ell$-torsion approximation $(F,A)$ of $(J_\nu\,;\,\nu\in S)$ with $F$ being $K'$-linearly disjoint to $K'(\mu_\ell)$ has big image at $\ell$.
\end{prop}
\begin{proof}
By the appendix of \cite{Hall}, there exists a hyperelliptic curve of genus $n$ above $K$ whose Jacobian $J$ has trivial endomorphism ring over $\overline{K}$ and satisfies Hall's condition
\begin{eqnarray*}
(T)&\quad:\quad&\parbox[l]{12cm}{\textup{There exists a finite extension $L/K$ such that the N\'eron model of $J\otimes_KL$ over $\O_L$ has a semi-stable fiber with toric dimension one.}}
\end{eqnarray*}
This property is preserved under finite extensions of $K$. By Theorem 1 of \cite{Hall}, there exists a natural number $\ell_0$ such that for any prime number $\ell\geq\ell_0$, $J\otimes_KK(J[j])$ has big image at $\ell$. Then $J$ has big image at $\ell$ as well, and the fields $K(J[\ell])$ and $K(J[j])$ are $K$-linearly disjoint. After enlarging $\ell_0$, we may in addition assume that $K/\Q$ is unramified at all primes $\ell\geq\ell_0$; then $[K(\mu_\ell):K]=\ell-1$ for all primes $\ell\geq\ell_0$, and hence $\ell_0$ satisfies the conclusion of Lemma \ref{mainsurjprop} for $K$. By Proposition \ref{prop:appendix} of the appendix to this paper, we may, after possibly further enlarging $\ell_0$ (so that $\ell_0\geq \max\{5, p:p\vert n\}$), assume that for every prime $\ell\geq\ell_0$, $\GSp(J[\ell](\overline{K}))$ contains an irreducible element $\mathfrak{s}_\ell$ of nonzero trace and a transvection $\mathfrak{t}_\ell$. Let us now fix a prime number $\ell\geq\ell_0$, let us set $K'':=K(J[\ell],J[j])$, and let $K'/K$ be a finite field extension that is linearly disjoint from $K''/K$; then $J\otimes_KK'$ has big image at $\ell$, and $[K'(\mu_\ell):K']=\ell-1$. By Lemma \ref{cebotarevlem}, there exist finite places $\nu_s$ and $\nu_t$ of $K'$ away from any given finite set of finite places of $K'$ such that the image of $J_{\nu_s}:=J\otimes_KK'_{\nu_s}$ at $\ell$ contains $\mathfrak{s}_\ell$, such that the image of $J_{\nu_t}:=J\otimes_KK'_{\nu_t}$ at $\ell$ contains $\mathfrak{t}_\ell$ and such that both $J_{\nu_s}$ and $J_{\nu_t}$ admit a full symplectic level $j$ structure. Indeed, $K'(J[\ell])$ and $K'(J[j])$ are $K'$-linearly disjoint because $K(J[\ell])$ is $K$-linearly disjoint to $K(J[j])$ and because $K'$ is $K$-linearly disjoint to $K''=K(J[\ell],J[j])$. Let us set $S:=\{\nu_s,\nu_r\}$, and let $(J_\nu\,;\,\nu\in S)$ be the resulting local AV-datum with level $j$ structure. Let now $(F,A)$ be an $\ell$-torsion approximation of $(J_\nu\,;\,\nu\in S)$ such that $F$ and $K'(\mu_\ell)$ are $K'$-linearly disjoint. Then $[F(\mu_\ell):F]=\ell-1$, so $F$ satisfies the conclusion of Lemma \ref{mainsurjprop}, and the image of $A$ at $\ell$ contains the symplectic group by Theorem \ref{Principal}. It follows that $A$ has big image at $\ell$, as desired.
\end{proof}

Let us note that we can always take $K'=K$ in the statement of Proposition \ref{localsurjprop}. However, it will prove useful later to have the full strength of Proposition \ref{localsurjprop} at one's disposal, i.e.\ to be able to obtain a local AV-datum as in the conclusion of  Proposition \ref{localsurjprop} over a rather general finite field extension $K'$ of $K$.

\begin{rem}\label{constantremark}
The constant $\ell_0$ from Proposition \ref{localsurjprop} is explicit, and it depends only on $K$, $n$ and $j$. Indeed, given $K$ and $n$, we can fix a hyperelliptic curve $C$ of genus $n$ defined over $K$ with trivial endomorphism ring and satisfying Hall's condition (cf. Appendix to \cite{Hall}). Denote by $J_C$ its Jacobian. Fix $j\geq 1$; there are explicit formulas for the field $K(J_C[j])$ (cf. \cite{Cantor1994}). Let $p_0$ be the biggest prime number that ramifies in $K(J_C[j])$. Theorem 1 of \cite{Hall} provides an explicit constant $\ell'_0$ depending only on $J_C$ and $K(J_C[j])$ such that 
for all $\ell\geq \ell'_0$, $\rho_{J, \ell}:G_{K(J_C[j])}\rightarrow \GSp_{2n}(\mathbb{F}_{\ell})$ is surjective. In the proof of Proposition \ref{localsurjprop} above, we have taken
\[
\ell_0=\max\{\ell'_0, p_0, 5, p: p\vert n\}\;;
\]
hence our constant $\ell_0$ is effective.
\end{rem}

\subsection{Local conditions enforcing tameness}

\begin{prop}\label{localtameprop}
Let $K$ be a number field, let $n,j\geq 1$ be integers, let $S,S'$ be finite disjoint sets of finite places of $K$, and let $\ell$ be a prime avoiding both $j$ and $S'$. There exists a finite extension $K'/K$ together with an $n$-dimensional local AV-datum with level $j$ structure $(A_\nu\,;\,\nu\in T)$ over $K'$, where $T$ denotes the set of places of $K'$ over $S$, such that $K'/K$ is unramified in the places above $\ell$, such that $K'/K$ is totally ramified over $S'$ and such that the $\ell$-torsion representations of the $A_\nu$ are tamely ramified.
\end{prop}
\begin{proof}
We may assume that $S$ contains all places above $\ell$. For every place $\nu\in S$, let us choose  an elliptic curve $E_\nu$ over $\mathbb{Q}$ such that $E_\nu$ has good reduction and such that $E_\nu$ has good supersingular reduction whenever $\nu$ divides $\ell$; for the existence of these elliptic curves see for instance \cite{AriasVila1} Cor.\ 3.6 and Prop.\ 3.7. Let us then set $K_\nu':=K_\nu(E_\nu[j])$. If $\nu$ is a place above $\ell$, the N\'eron-Ogg-Shafarevich criterion implies that the extension $K_{\nu}'/K_{\nu}$ is unramified (recall that $\ell\nmid j$).

After possibly enlarging the $K'_\nu$ by means of  unramified extensions, we may assume that the degrees $[K_\nu':K_\nu]$ all coincide; let $r$ denote this common degree. For each $\nu\in S'$, let $K'_\nu/K_\nu$ be a totally ramified extension of degree $r$, obtained for instance by extracting an $r$-th root of a uniformizer of $K_\nu$. There exists a finite extension $K'/K$ of degree $r$ such that the induced local extensions at $S\cup S'$ coincide with the $K'_\nu/K_\nu$: indeed, for each $\nu\in S\cup S'$, let $\alpha_\nu$ be a primitive element for $K'_\nu/K_\nu$, and let $f_\nu\in K_\nu[X]$ be its minimal polynomial. By the weak approximation theorem (cf.\ \cite{bourbaki_ca} Chap.\ VI \S 7 No.\ 3 Thm.\ 2), there exists a monic polynomial $f\in K[X]$ of degree $r$ which approximates the $f_\nu$  simultaneously, up to a precision such that Krasner's Lemma (cf.\ \cite{blr} 3.4.2 Prop.\ 3 and Cor.\ 4) applies; we then set $K'=K[x]/(f)$. By construction, $K'/K$ has the desired ramification behavior. For each $\nu\in S$, let us consider the abelian variety $A_\nu=E_\nu^n$ over $K_\nu'$. Then $A_\nu$ admits a full symplectic level $j$ structure, and furthermore, the $\ell$-torsion representation of $A_\nu$ is tamely ramified. For $\nu|\ell$, this follows from \cite{Serre1972} Prop.\ 13, and for $\nu\in S$ not dividing $\ell$, the N\'eron-Ogg-Shafarevich criterion (cf.\ \cite{serre-tate} Theorem 1) even guarantees that $\rho_{A_\nu,\ell}$ is unramified.
\end{proof}

\begin{rem}
In the Situation of Proposition \ref{localtameprop}, if $S$ is the set of places of $K$ lying above the set of rational primes dividing the order of $\GSp_{2n}(\F_\ell)$ and if $(F,A)$ is an $\ell$-torsion approximation of $(A_\nu\,;\,\nu\in T)$, then the $\ell$-torsion representation of $A$ is everywhere tamely ramified.
\end{rem}

\subsection{Existence of $\ell$-torsion approximations with big image}

\begin{prop}\label{prop:step2}
Let $K$ be a number field, and let $n\geq 1$, $j\geq 3$ be integers such that $\mu_j(\overline{K})\subseteq K$. Then there exists a constant $\ell_0$ such that for every prime number $\ell\geq \ell_0$, there exists a finite extension $K''/K$ such that for every finite extension $K'/K$ that is linearly disjoint to $K''/K$, every $n$-dimensional local AV-datum with level $j$ structure over $K'$ admits an $\ell$-torsion approximation with big image at $\ell$.
\end{prop}

\begin{proof}
Let us choose $\ell_0$ as in Proposition \ref{localsurjprop}, let $\ell\geq\ell_0$ be a prime number, let $K''/K$ be as in Proposition \ref{localsurjprop}, let $K'/K$ be a finite extension that is linearly disjoint to $K''/K$, and let $(A_\nu\,;\,\nu\in S)$ be an $n$-dimensional local AV-datum with level $j$ structure over $K'$. By Proposition \ref{localsurjprop}, there exists an $n$-dimensional local AV-datum with level $j$ structure $(J_\nu\,;\,\nu\in S')$ over $K'$ such that $S'$ is disjoint to $S$ and such that the conclusion of Proposition \ref{localsurjprop} holds. For $\nu\in S'$, let us write $A_\nu:=J_\nu$. By Theorem \ref{thm:step1}, there exists an $\ell$-torsion approximation $(F,A)$ of $(A_\nu\,;\,\nu\in S\cup S')$ such that $F$ is linearly disjoint to $K'(\mu_\ell)$ over $K'$. Then $(F,A)$ is an $\ell$-torsion approximation of both $(A_\nu\,;\,\nu\in S)$ and $(J_\nu\,;\,\nu\in S')$; by Proposition \ref{localsurjprop}, it follows that $A$ has big image at $\ell$.
\end{proof}

Let us note that in the statement of Proposition \ref{prop:step2}, we can always choose $K'=K$.

\section{Proof of the main result}\label{proofsec}




We can now give the proof of Theorem \ref{mainthm}. Let us restate it:
\begin{thm}\label{mainthm2}
Given a number field $K$ and an integer $n \geq 1$, there exists an integer $\ell_0$ such
that for all prime numbers $\ell\geq \ell_0$, there exist a finite extension $F$ of $K$, unramified in all places above $\ell$, and an $n$-dimensional abelian variety $A$ defined over $F$ such that the $\ell$-torsion representation of $A$ is surjective and everywhere tamely ramified.
\end{thm}
\begin{proof}
Let us fix any $j\geq 3$. If $K'/K$ is a finite extension and if $\ell_0'$ is a constant such that the statement of the theorem holds for $K'$ and $\ell_0'$, then the theorem holds for $K$ and any constant $\ell_0\geq\ell_0'$ such that $K'/K$ is unramified in all places above the rational primes $\ell\geq\ell_0$. We may thus replace $K$ by a finite extension and hereby assume that $\mu_j(\overline{K})\subseteq K$. Let $\ell_0$ be the constant given by Proposition \ref{prop:step2}. After possibly enlarging $\ell_0$, we may assume that $\ell_0>j$; then $j$ is coprime to any prime number $\ell\geq\ell_0$. Let us fix a prime number $\ell\geq\ell_0$, let $K''/K$ be the finite extension given by Proposition \ref{prop:step2}, let $S$ denote the set of places of $K$ dividing the order of $\GSp_{2n}(\F_\ell)$; then $S$ contains all places above $\ell$. Let moreover $S'$ be the set consisting of a single finite place $\mu$ of $K$ away from $S$ such that $K''/K$ is unramified at $\mu$. Let $K'/K$ and $(A_\nu\,;\,\nu\in T)$ be the associated data given by Proposition \ref{localtameprop}; then $K'/K$ is totally ramified at $\mu$, while $K''/K$ is unramified in $\mu$, and it follows from \cite{fried_jarden} Lemma 2.5.8 that $K'$ and $K''$ are linearly disjoint over $K$. By Proposition \ref{prop:step2}, there exists an $\ell$-torsion approximation $(F,A)$ of $(A_\nu\,;\,\nu\in T)$ with big image at $\ell$. Now $F/K'$ is totally split over $T$, $K'/K$ is unramified in the places above $\ell$, and $T$ contains all places of $K'$ above $\ell$. The pair $(F,A)$ thus has the desired properties.
\end{proof}

Our methods can be used to prove the following strengthening of Theorem \ref{mainthm}:
\begin{thm}
Let $K$ be a number field, let $n\geq 1$ be an integer, and let $\ell_0$ be the constant given by Theorem \ref{mainthm2} and its proof, for some $j\geq 3$. Then for all primes $\ell\geq \ell_0$, the pairs $(F, A)$ satisfying the conclusion of Theorem \ref{mainthm} lie Zariski-dense in the moduli-space $\sA=\sA_{n,1,j/K}$ (cf.\ Remark \ref{someremarks} ($ii$)). In particular, there exist infinitely many pairwise geometrically non-isomorphic such pairs.
\end{thm}
\begin{proof}
Indeed, let us assume that all these points lie on a proper closed subvariety $V\subsetneq \sA$, and let $U$ denote the complement of $V$ in $\sA$. We argue exactly as above, except that in the proof of Theorem \ref{thm:step1}, we apply Theorem \ref{mbthm} to $U$ and to the intersections
\[
\Omega_\nu'\,:=\,\Omega_\nu\cap U(K_\nu)\;;
\]
Theorem \ref{mbthm} applies to this input data because $U$ is smooth and geometrically irreducible and because the sets $\Omega'_\nu$ are again open and non-empty, where non-emptiness follows by dimension reasons from the openness of the $\Omega_\nu$. The pair $(F,A)$ that is produced by the proof of Theorem \ref{mainthm2} then defines an $F$-valued point of $U$, contrary to our assumption.
\end{proof}

\begin{rem}
One may wonder whether for a fixed finite extension $F/K$ that is unramified in the places above $\ell$, the pairs $(F,A)$ satisfying the conditions of Theorem \ref{mainthm2} lie dense in the moduli, or one may ask the weaker question whether there exist infinitely many pairwise non-isomorphic such pairs. At present we do not know how to obtain such a result.
\end{rem}

\section{Appendix}

\begin{defn} Let $q$ be a power of a prime $p$, and let
$\mathbb{F}_q$ be the field with $q$ elements. Let $f\in \F_q[x]$ be a monic polynomial. 
\begin{enumerate}
\item Write $f(x)=x^r + a_{r-1}x^{r-1}+ \cdots + a_1 x + a_0$. We say that $a_{r-1}$ is the \emph{trace} of $f(x)$.
\item Assume the degree of $f$ is $2n$ for some $n\in \mathbb{N}$. We say that $f$ is \emph{symplectic} if it satisfies that $a_i=a_{2n-i}$ for all $i=1, \dots, n$ and $a_0=1$, that is, if $f(x)$ has the shape

$$x^{2n} + a_1 x^{2n-1} + \cdots + a_{n-1} x^{n+1} + a_n x^n + a_{n-1}x^{n-1} + \cdots + a_1 x + 1.$$
It is easily seen that $f(x)$ is symplectic if and only if it
satisfies the relation $$x^{2n}f\left(\frac{1}{x}\right)=f(x)$$ in
the field $\F_q(x)$.
 \end{enumerate}
\end{defn}

In this appendix, we give a proof of the following result:

\begin{prop}\label{prop:appendix} For any positive integer
$n\geq 1$, for all prime numbers $p\nmid n$ and for all $r\in\N$ such that $p^r\geq 5$, the ring $\F_{p^r}[X]$ contains an
irreducible symplectic polynomial of nonzero trace and of degree
$2n$.
\end{prop}

The proof will follow from a series of elementary lemmas.

\begin{lem}\label{lem:relation} Let $q$ be a prime power, let $f(x)=x^n + a_{n-1}x^{n-1} + \cdots + a_1 x + a_0\in
\mathbb{F}_q[x]$ be a monic irreducible polynomial,
and let $\alpha\in \overline{\mathbb{F}}_{q}$ be a root. Let $\beta$ be a root of $x^2 - \alpha x + 1$, and let us assume that $\beta\not\in \mathbb{F}_q(\alpha)$.
Then the minimal polynomial of $\beta$ over $\mathbb{F}_q$ is symplectic, and 
\begin{equation*}\mathrm{tr}_{\mathbb{F}_q(\alpha)/\mathbb{F}_q}(\alpha)=\mathrm{tr}_{\mathbb{F}_q(\beta)/\mathbb{F}_q}(\beta).\end{equation*}
\end{lem}

\begin{proof} Let us note that $\alpha=\beta + \frac{1}{\beta}$, and let us consider the polynomial $g(x)=x^nf\left(x + \frac{1}{x}\right)\in
\mathbb{F}_q[x]$; then $g(x)$ is a symplectic polynomial satisfying $g(\beta)=0$. Since
$\mathbb{F}_q(\beta)$ has degree $2n$ over $\mathbb{F}_q$ and
since $g(x)$ is a monic polynomial of degree $2n$, $g(x)$ must be the minimal
polynomial of $\beta$ over $\mathbb{F}_q$. Therefore
$\mathrm{tr}_{\mathbb{F}_q(\beta)/\mathbb{F}_q}(\beta)$ is equal to
the coefficient of $x^{2n-1}$ in $g(x)$, that is to say, $a_{n-1}$,
which is precisely
$\mathrm{tr}_{\mathbb{F}_q(\alpha)/\mathbb{F}_q}(\alpha)$.
\end{proof}

To prove Proposition \ref{prop:appendix}, it now suffices to show that for $p$ not dividing $n$ and for $q=p^r\geq 5$, we can find an $\alpha\in \overline{\mathbb{F}}_{q}$ with
$\mathbb{F}_q(\alpha)/\mathbb{F}_q$ of degree $n$ and nonzero trace such that the polynomial $x^2-\alpha x + 1$ is irreducible.

\begin{lem}\label{lem:Bound1} If $p$ is a prime number, if $q$ is a power of $p$ and if $n\geq 1$ such that $p\nmid n$, then the number of monic irreducible polynomials in $\mathbb{F}_q[x]$ of degree $n$ and nonzero
trace is equal to
\begin{equation*} \frac{q-1}{qn}\sum_{d\vert n}\mu(d) q^{\frac{n}{d}}.\end{equation*}
\end{lem}

 \begin{proof} The number of monic irreducible  polynomials of degree
$n$ is $\frac{1}{n}\sum_{d\vert n}\mu(d) q^{\frac{n}{d}}$ (cf. Theorem 3.25 of
\cite{LidlNiederreiter}). On the other hand, we can define an equivalence relation on the set of monic irreducible polynomials of degree $n$ in $\mathbb{F}_q[x]$ by declaring that $f\equiv g$ if and only if there exists an $a\in \mathbb{F}_q$ with $f(x)=g(x-a)$. Each equivalence class consists of precisely $q$ elements, and the traces of the representatives of any given class are all distinct. Hence for each $a\in \mathbb{F}_q$, the cardinality of the set of monic irreducible polynomials with trace equal to $a$ is $\frac{1}{qn}\sum_{d\vert n}\mu(d) q^{\frac{n}{d}}$.
\end{proof}

\begin{lem}\label{lem:Bound2}
Let $q$ be a prime power, and let $n$ be a positive integer; then the number of elements $\alpha\in \mathbb{F}_{q^n}$ such that $x^2-\alpha x + 1$ is reducible over $\mathbb{F}_{q^n}$ equals
\[
\begin{cases}
\frac{q^n+1}{2}&\textup{if $q$ is odd and}\\
\;\;\,\frac{q^n}{2}\;\;\,&\textup{if $q$ is even.}\;
\end{cases}
\]
\end{lem}

\begin{proof} Let $\alpha$ be any element of $\mathbb{F}_{q^n}$; then the polynomial $x^2-\alpha x + 1$ is reducible over $\mathbb{F}_{q^n}$ if and only if $\alpha=\beta + \frac{1}{\beta}$ for some $\beta\in \mathbb{F}_{q^n}$.
Let us consider the map 
\[
\phi:\mathbb{F}_{q^n}^{\times}\rightarrow \mathbb{F}_{q^n}\quad;\quad\beta\mapsto\beta + \frac{1}{\beta}\;;
\]
we have to compute the cardinality of its image. To do so, we compute the cardinalities of the fibers of $\phi$: let  us consider an element $\alpha\in\F_{q^n}$ that lies in the image of $\phi$, i.e.\ for which the quadratic equation
\[
f_\alpha(x)\,=\,x^2-\alpha x+1\,\in\mathbb{F}_{q^n}[x]
\]
has a root $\beta$ in $\F_{q^n}$; then $1/\beta$ is also a root of the above equation. Let us note that $\beta=1/\beta$ if and only if $\beta=\pm 1$. Hence, if $\beta\neq\pm 1$, the cardinality of $\phi^{-1}(\alpha)$ is $2$, since $f_\alpha(x)$ can have at most two different roots. On the other hand, if $\beta=\pm 1$, then the cardinality of $\phi^{-1}(\alpha)$ is $1$, because $\beta$ is then a multiple root of $f_\alpha(x)$: indeed, then
\[
f'_\alpha(x)\,=\,2x-\alpha\,=\,2x-(\beta+\frac{1}{\beta})=2x-2\beta
\]
vanishes in $\beta$. Let us moreover note that if $q$ is odd, then $\phi(1)=2$ is different from $\phi(-1)=-2$. Since $\F_{q^n}^\times$ is the disjoint union of the fibers of $\phi$, we conclude that for $q$ odd,
\[
q^n-1\,=\,2\cdot(|\textup{im}(\phi)|-2)+1+1\;,
\]
whereas for $q$ even,
\[
q^n-1\,=\,2\cdot(|\textup{im}(\phi)|-1)+1\;;
\]
the claim now follows by a straightforward computation.
\end{proof}

From Lemmas \ref{lem:Bound1} and \ref{lem:Bound2} we obtain the following result.

\begin{lem}\label{lem:Bound3} Let us assume that $p\nmid n$; then the number of $\alpha\in \overline{\mathbb{F}}_{q}$ with
$\mathbb{F}_q(\alpha)/\mathbb{F}_q$ of degree $n$, nonzero trace and such that the polynomial $x^2-\alpha x + 1$ is irreducible over $\mathbb{F}_{q^n}$ is greater than or equal to 
\begin{equation*}
\frac{q-1}{q}\sum_{d\vert n}\mu(d) q^{\frac{n}{d}} -\frac{q^n+1}{2}
\end{equation*}
\end{lem}
\qed
\begin{proof}[Proof of Proposition \ref{prop:appendix}] Let us now assume that $q\geq 5$ and that $p\nmid n$.  Combining  Lemmas \ref{lem:relation} and \ref{lem:Bound3}, we see that it suffices to prove that the number 
\begin{equation*}
M:=\frac{q-1}{q}\sum_{d\vert n}\mu(d) q^{\frac{n}{d}} -\frac{q^n+1}{2}
\end{equation*} is positive. We distinguish two cases. First, if $n=1$, then 
\begin{equation*}M=\frac{q-1}{q}q -\frac{q+1}{2}=\frac{1}{2}q-\frac{3}{2}>0 \end{equation*}
as desired. On the other hand, if $n>1$, then 
\begin{eqnarray*}
M&=&\frac{q-1}{q}\sum_{d\vert n}\mu(d) q^{\frac{n}{d}} -\frac{q^n+1}{2}\\
&=&\frac{q-1}{q}(q^n + \sum_{d\vert n\atop d\not=1}\mu(d) q^{\frac{n}{d}}) -\frac{q^n+1}{2}\\
&=&\frac{1}{2}(q^n-1) -q^{n-1} + \frac{q-1}{q}\sum_{d\vert n\atop d\not=1}\mu(d) q^{\frac{n}{d}}\\
&\geq&\frac{1}{2}(q^n-1) -q^{n-1}-\frac{q^n-1}{q-1}\\
&\geq& \frac{1}{4}(q^n-1) -  q^{n-1}\;,\\
&>&0
\end{eqnarray*}
where for the third last inequality we used the inequalities $\mu(d)\geq -1$ and $(q-1)/q\leq 1$ as well as the geometric series.
\end{proof}

\bibliographystyle{plain}
\bibliography{tameramabvar}

\end{document}